\documentclass[11pt,twoside,a4paper]{article}

\author{Shai Sarussi}
\title{Quasi-valuations and algebras over valuation domains}
\date{}

\usepackage{amsmath,amsthm,amssymb}
\usepackage{verbatim}

\begin{document}

\newtheorem{thm}{Theorem}[section]
\newtheorem{cor}[thm]{Corollary}
\newtheorem{lem}[thm]{Lemma}
\newtheorem{prop}[thm]{Proposition}
\newtheorem{ax}{Axiom}

\theoremstyle{definition}
\newtheorem{defn}[thm]{Definition}

\theoremstyle{remark}
\newtheorem{rem}[thm]{Remark}
\newtheorem{ex}[thm]{Example}
\newtheorem*{notation}{Notation}

\newcommand{\qv}{{quasi-valuation\ }}


\maketitle

\begin{abstract} {Suppose $F$ is a field with valuation $v$ and
valuation domain $O_{v}$, and $R$ is an $O_{v}-$algebra.
We prove that $R$ satisfies SGB (strong going between) over $O_{v}$.
We give a necessary and sufficient condition for $R$ to satisfy LO (lying over) over $O_{v}$.
Using the filter \qv constructed in [Sa1], we show that if $R$ is torsion-free
over $O_{v}$ then $R$ satisfies GD (going down) over $O_{v}$. In particular, if $R$ is torsion-free and $(R^{\times} \cap O_{v}) \subseteq O_{v}^{\times}$,
then for any chain in $\text{Spec}(O_v)$ there exists
a chain in $\text{Spec}(R)$ covering it.

Assuming $R$ is torsion-free over $O_{v}$ and $[R \otimes_{O_{v}}F:F]< \infty$, we prove that $R$ satisfies INC (incomparabilty) over $O_{v}$.
Assuming in addition that $(R^{\times} \cap O_{v}) \subseteq O_{v}^{\times}$, we deduce that $R$ and $O_{v}$ have the same Krull
dimension and a bound on the size of the prime spectrum of $R$ is given.

Under certain assumptions on $R$ and a \qv defined on it, we prove that the \qv ring satisfies GU (going up) over $O_{v}$.
Combining these five properties together, we deduce that any maximal chain of prime
ideals of the \qv ring is lying over $\text{Spec}(O_{v})$, in a one-to-one correspondence.




}\end{abstract} 
\section{Introduction}

Recall that a valuation on a field $F$ is a function $v : F
\rightarrow \Gamma \cup \{  \infty \}$, where $\Gamma$ is a
totally ordered abelian group and where $v$ satisfies the following
conditions:


 (A1) $v(x) \neq \infty$ iff $  x \neq 0$, for all $x\in F$;

(A2) $v(xy) = v(x)+v(y)$ for all $x,y \in F$;

(A3) $v(x+y) \geq \min \{ v(x),v(y) \}$ for all $x,y \in F$.

Recall (cf. [Sa1, Introduction]) that a quasi-valuation on a ring $R$ is a function $w
: R \rightarrow ~M \cup \{  \infty \}$, where $M$ is a totally
ordered abelian monoid, to which we adjoin an element $\infty$, which is
greater than all elements of $M$, and where $w$ satisfies the following
properties:

(B1) $w(0) = \infty$;

(B2) $w(xy) \geq w(x) + w(y)$ for all $x,y \in R$;

(B3) $w(x+y) \geq \min \{ w(x), w(y)\}$ for all $x,y \in R$.

In [Sa1] we developed the theory of quasi-valuations that extend a
given valuation. This paper continues the study of
quasi-valuations, which are natural generalizations of
valuations. Other related theories, including pseudo-valuations
(see [Co],[Hu] and [MH]), Manis-valuations and PM-valuations (see
[KZ]), value functions (see [Mo]), and gauges (see [TW]), are
discussed briefly in the introduction of [Sa1].

In this paper we shall generalize some of the results in [Sa1], as
well as studying some new concepts concerning algebras over
valuation domains. The proofs of some of the results to be
generalized use the same methods as in the commutative case, with
adjustments to the non-commutative case. However, most proofs require totally different approaches.

{\it In this paper $F$ denotes a field with a nontrivial valuation
$v$ and valuation domain $O_{v}$, $R$ is an algebra over $O_{v}$, and
$w:R \rightarrow M \cup \{ \infty \}$ is a quasi-valuation on $R$.
From time to time we consider a commutative valuation ring $S$, which is not necessarily a valuation domain,
and an $S-$algebra $R$.}


We list here some of the common symbols we use for $v$ a valuation
on a field $F$, and $w$ a quasi-valuation on a ring $R$ (usually
$R$ is taken to be an $O_{v}-$algebra and $w$ is a $v-$quasi-valuation):

$O_{v}=\{x\in F \ | \ v(x)\geq0\}$; the valuation domain.

$I_{v}=\{x\in F \ | \ v(x)>0\}$; the valuation ideal.

$O_{w}=\{x\in R \ | \ w(x)\geq0\}$; the quasi-valuation ring.

$I_{w}=\{x\in R \ | \ w(x)>0\}$; the quasi-valuation ideal.

$\Gamma_{v}$; the value group of the valuation $v$.

$M_{w}$; the value monoid of the quasi-valuation $w$, i.e., the
submonoid of $M$ generated by $  w(R\setminus \{ 0 \})$.

\begin{rem} \label{remark on sa1} From time to time we discuss in this paper some of the results presented in [Sa1]. In order to
avoid repetitions, we recall now the main scope discussed in [Sa1], to which we usually refer:
Let $F$ be a field with a nontrivial valuation
$v$ and valuation domain $O_{v}$, and let $E/F$ be a finite dimensional field extension.
Let $w:E \rightarrow M \cup \{ \infty \}$ be a quasi-valuation on $E$ extending $v$ on $F$
with \qv ring $O_{w}$. We will note whenever an additional assumption was added in [Sa1].
\end{rem}


In this paper the symbol $\subset$ means proper inclusion and the symbol $\subseteq$ means inclusion
or equality.

\subsection{Previous results - the construction of the filter \qv}




For the reader's convenience, we recall from [Sa1] the main steps
in constructing the filter quasi-valuation. For further details
and proofs, see [Sa1, Section 9].

The first step is to construct a value monoid, constructed from
the value group of the valuation. We call this value monoid the
cut monoid. We start by reviewing some of the basic
notions of Dedekind cuts of ordered sets. For further information
on Dedekind cuts see, for example, [AKK] or [Weh].

\begin{defn} Let $T$ be a totally ordered set. A subset $S$ of $T$ is
called initial if for every $\gamma \in S$ and $\alpha \in T$, if
$\alpha \leq \gamma$ then $\alpha \in S$. A cut $\mathcal
A=(\mathcal A^{L}, \mathcal A^{R})$ of $T$ is a partition of $T$
into two subsets $\mathcal A^{L}$ and $\mathcal A^{R}$, such that,
for every $\alpha \in \mathcal A^{L}$ and $\beta \in \mathcal
A^{R}$, $\alpha<\beta$. \end{defn}

 The set of all cuts $\mathcal A=(\mathcal A^{L}, \mathcal A^{R})$
 of the ordered set $T$ contains the two cuts $(\emptyset,T)$ and
 $(T,\emptyset)$; these are commonly denoted by $-\infty$ and
 $\infty$, respectively. However, we do not use the symbols $-\infty$ and
 $\infty$ to denote the above cuts since we
 define a "different" $\infty$.

Given $\alpha \in T$, we denote
$$(-\infty,\alpha]=\{\gamma \in T \mid \gamma \leq \alpha \}$$ and
$$(\alpha,\infty)=\{\gamma \in T \mid \gamma > \alpha \}.$$
One defines similarly the sets $(-\infty,\alpha)$ and
$[\alpha,\infty)$.

To define a cut we often write $\mathcal A^{L}=S$, meaning that
$\mathcal A$ is defined as $(S, T \setminus S)$ when $S$ is an
initial subset of $T$. The ordering on the set of all cuts of $T$
is defined by $\mathcal A \leq \mathcal B$ iff $\mathcal A^{L}
\subseteq \mathcal B^{L}$ (or equivalently $\mathcal A^{R}
\supseteq \mathcal B^{R}$). Given $S  \subseteq T$, $S^{+}$ is the smallest cut $\mathcal A$ such
that $S  \subseteq \mathcal A^{L}$. In particular, for $\alpha \in T$ we have
$\{\alpha\}^{+}=((-\infty,\alpha],(\alpha,\infty))$.

For a group $\Gamma$, subsets $S,S' \subseteq \Gamma$ and $n \in
\Bbb N$, we define

$$S+S'=\{ \alpha+\beta \mid \alpha \in S, \beta \in S' \};$$

$$nS=\{ s_{1}+s_{2}+...+s_{n} \mid s_{1},s_{2},...,s_{n} \in S
\}.$$

Now, for $\Gamma$ a totally ordered abelian group, $\mathcal
M(\Gamma)$ is called the cut monoid of $\Gamma$.
$\mathcal M(\Gamma)$ is a totally ordered abelian monoid.

For $\mathcal A , \mathcal B \in \mathcal M(\Gamma)$, their (left)
sum is the cut defined by $$(\mathcal A + \mathcal B)^{L}=\mathcal
A^{L} + \mathcal B^{L}.$$ The zero in $\mathcal M(\Gamma)$ is the
cut $((-\infty,0],(0,\infty))$.

 For $\mathcal A \in \mathcal M(\Gamma)$ and
$n \in \Bbb N$,  the cut $n\mathcal A$ is defined by $$(n\mathcal
A)^{L}=n\mathcal A^{L}.$$

Note that there is a natural monomorphism of monoids $\varphi :
\Gamma \rightarrow \mathcal M(\Gamma)$ defined in the following
way: for every $\alpha \in \Gamma$, $$\varphi (\alpha) =
((-\infty,\alpha],(\alpha,\infty)) $$

For $\alpha \in \Gamma$ and $\mathcal B \in \mathcal M(\Gamma)$,
we denote $\mathcal B-\alpha$ for the cut $\mathcal B+(-\alpha)$
(viewing $-\alpha$ as an element of $\mathcal M(\Gamma)$).


\begin{defn} (cf. [Sa1, Definition 9.13]) Let $v$ be a valuation on a field $F$ with value group
$\Gamma_{v}$. Let $O_{v}$ be the valuation domain of $ v$ and let
$R$ be an algebra over $O_{v}$. For every $x \in R$, the
$O_{v}$-{\it support} of $x$ in $R$ is the set
$$S^{R/O_{v}}_{x}=\{ a \in O_{v} | xR \subseteq aR\}.$$ We suppress
$R/O_{v}$ when it is understood. \end{defn}

For every $A \subseteq O_{v}$ we denote $(v(A))^{\geq 0}=\{ v(a)
\mid a \in A \}$; in particular,

$$(v(S_{x}))^{\geq 0}=\{ v(a) | a \in S_{x} \}.$$ The reason for
this notion is the fact that $(v(S_{x}))^{\geq 0}$ is an initial
subset of $\ (\Gamma_{v}) ^{\geq 0}$ ($=\{ \gamma \in \Gamma_{v} \mid \gamma \geq 0\}$).

We define $$v(S_{x})=(v(S_{x}))^{\geq 0} \cup  \{ \gamma \in \Gamma_{v} \mid \gamma < 0\};$$ 
and note that $v(S_{x})$ is an initial subset of $\Gamma_{v}$.

Note that if $A$ and $ B$ are subsets of $O_{v}$ such that $A
\subseteq B$ then $v(A) \subseteq v(B)$.

Recall that we do not denote the cut $(\Gamma_{v},\emptyset) \in
\mathcal M(\Gamma_{v}) $ as $\infty$. So, as usual, we adjoin to
$\mathcal M(\Gamma_{v}) $ an element $\infty$ greater than all
elements of $\mathcal M(\Gamma_{v})$; for every $\mathcal A \in
\mathcal M(\Gamma_{v})$ and $\alpha \in \Gamma_{v}$ we define
$\infty+\mathcal A=\mathcal A+\infty=\infty$ and
$\infty-\alpha=\infty$.

\begin{defn} (cf. [Sa1, Definition 9.17]) Let $v$ be a valuation on a field $F$
with value group $\Gamma_{v}$. Let $O_{v}$ be the valuation domain
of $v$ and let $R$ be an algebra over $O_{v}$. Let $\mathcal
M(\Gamma_{v})$ denote the cut monoid of $\Gamma_{v}$. We
say that a function $w:R \rightarrow \mathcal M(\Gamma_{v})
\cup \{ \infty \}$ is {\it induced by} $(R,v)$ if $w$ satisfies
the following:

 1. $w(x)=(v(S_{x}),\Gamma_{v} \setminus v(S_{x}))$ for every $0 \neq x \in
 R$. I.e.,  $w(x)^{L}=v(S_{x})$;

 2. $w(0)=\infty$.


\end{defn}


\begin{thm} (cf. [Sa1, Theorem 9.19]) \label{Sa1, Theorem 9.19} Let $v$ be a valuation on a field $F$ with value group
$\Gamma_{v}$. Let $O_{v}$ be the valuation domain of $\ v$ and let
$R$ be an algebra over $O_{v}$. Let $\mathcal M(\Gamma_{v})$
denote the cut monoid of $\Gamma_{v}$. Then there exists
a quasi-valuation $w:R \rightarrow \mathcal M(\Gamma_{v}) \cup \{
\infty \}$ induced by $(R,v)$.\end{thm}

 The quasi-valuation discussed in Theorem \ref{Sa1, Theorem 9.19} is called the
filter quasi-valuation induced by $(R,v)$.

The following lemma is very important for our study. We shall use it in Theorem
\ref{R torsion-free over Ov satisfies GD} to prove that any torsion-free
$O_{v}-$algebra satisfies GD over $O_{v}$. We shall also use it in Lemma \ref{minimal set with 1}.

\begin{lem} (cf. [Sa1, Lemma 9.25]) \label{Sa1, Lemma 9.25} Notation as in Theorem
\ref{Sa1, Theorem 9.19}, assume in addition that
$R$ is torsion-free over $O_{v}$; then
$$w(cx)=v(c)+w(x)$$ for every $c \in O_{v}$, $x \in R$.

\end{lem}

We note that even in the case where $R$ is a torsion-free algebra
over $O_{v}$, one does not necessarily have $w(c \cdot
1_{R})=v(c)$ for $c \in O_{v}$. This is despite the fact that
$$w(c \cdot 1_{R}) = v(c)+w( 1_{R})$$ by Lemma \ref{Sa1, Lemma 9.25}. The reason is that $w(
1_{R})$ is not necessarily $0$ (see, for example, [Sa1, Example 9.28]).

\begin{rem} Note that if $R$ is a torsion-free algebra over
$O_{v}$, then there is an embedding $R \hookrightarrow R
\otimes_{O_{v}}F$. In this case there exists a quasi-valuation on $R
\otimes_{O_{v}}F$ that extends the quasi-valuation on $R$. \end{rem}

\begin{lem} (cf. [Sa1, Lemma 9.31]) \label{[Sa1, Lemma 9.31]} Let $v, F, \Gamma_{v}$ and $O_{v}$ be as in Theorem
\ref{Sa1, Theorem 9.19}. Let $R$ be a torsion-free
algebra over $O_{v}$, $S$ a multiplicative closed subset of
$O_{v}$, $0 \notin S$, and let $w: R \rightarrow M \cup \{ \infty
\}$ be any quasi-valuation where $M$ is any totally ordered
abelian monoid containing $\Gamma_{v}$ and $w(cx)=v(c)+w(x)$ for
every $c \in O_{v}$, $x \in R$. Then there exists a
quasi-valuation $W$ on $R \otimes_{O_{v}}O_{v}S^{-1}$, extending
$w$ on $R$ (under the identification of $R$ with $R
\otimes_{O_{v}} 1$), with value monoid $ M \cup \{ \infty
\}$.\end{lem}

We recall that $W$ above is defined in the following way:
for all $r \otimes
\frac{1}{b} \in R \otimes_{O_{v}}O_{v}S^{-1}$,
$$W( r\otimes \frac{1}{b})=w(r)-v(b)\ (\ =w(r)+(-v(b))\ ).$$
It is not difficult to see that every element of $R \otimes_{O_{v}}O_{v}S^{-1}$ is of the form above (see [Sa1, Remark 9.29]).

\begin{thm} \label {Sa1, Theorem 9.32} (cf. [Sa1, Theorem 9.32]) Let $v, F, \Gamma_{v}, O_{v}$ and $\mathcal M
(\Gamma_{v})$ be as in Theorem \ref{Sa1, Theorem 9.19}. Let $R$ be a torsion-free algebra over $O_{v}$ and
let $w$ denote the filter quasi-valuation induced by $(R,v)$; then
there exists a quasi-valuation $W$ on $R \otimes_{O_{v}}F$,
extending $w$ on $R$, with value monoid $\mathcal M(\Gamma_{v})$ and $O_{W}=R \otimes_{O_{v}} 1$. \end{thm}

It is not difficult to see that for $W$, as in Theorem \ref{Sa1, Theorem 9.32}, $(\emptyset, \Gamma_{v}) \notin im(W).$ This $W$ is also
called the filter \qv induced by $(R,v)$.
For more information on quasi-valuations see [Sa1] and [Sa2].



Recall (cf. [End, page 47]) that an isolated subgroup of a totally ordered abelian
group $\Gamma$ is a subgroup $H$ of $\Gamma$, such that $\{ \gamma
\in \Gamma | 0 \leq \gamma \leq h \} \subseteq H$ for any $h \in
H$ (some texts call such subgroups "convex" or "distinguished").
Also recall that there is a one to one correspondence between the set
of all prime ideals of a valuation domain and the set
$G(\Gamma_{v})$ of all isolated subgroups of $\Gamma_{v}$ given by
$P \mapsto \{ \alpha \in \Gamma_{v} \mid \alpha \neq v(p)$ and
$\alpha \neq -v(p)$ for all $p \in P\}$. See [Sa1, Section 4] for more information on isolated subgroups and their
"corresponding objects" in a monoid containing $\Gamma$.

We shall use the terminology of the cut monoid $\mathcal M
(\Gamma_{v})$ introduced in [Sa1, Section 9]. We shall freely interchange
between the following two notions: for $\alpha \in \Gamma_{v}$ we
shall sometimes refer to it as an element of $\Gamma_{v}$ and
sometimes as an element of $\mathcal M (\Gamma_{v})$. For example,
let $\alpha \in \Gamma_{v}$ and let $H$ be an isolated
subgroup of $\Gamma_{v}$. If we consider $\alpha$ as an element of
$\Gamma_{v}$ we may write $\alpha \in H$ or $\alpha \notin H$; if
we consider $\alpha$ as an element of $\mathcal M (\Gamma_{v})$ we
may write $\alpha < H^+$ or $\alpha > H^+$, respectively.

\subsection{Basic definitions}



Let $S$ and $R$ be rings with $S$ commutative. It is well known that $R$ is an algebra over $S$ iff
there exists a unitary homomorphism $f: S \rightarrow R$ such that $f[S] \subseteq Z(R)$, where $Z(R)$ denotes the center of $R$.

Assume that $R$ is an algebra over $S$. For subsets $I \subseteq R$ and $J \subseteq S$ we say that $I$ is lying over $J$ if $J=\{s \in S \mid s \cdot 1_{R} \in I \}$. By abuse of notation, we shall write
$J=I \cap S$ (even when $R$ is not faithful over $S$).
Equivalently, $J=f^{-1}[I]$ for the unitary homomorphism $f: S \rightarrow R$ defined by $f(s)=s \cdot 1_{R}$.
We shall freely interchange between those two terminologies.
Note that if $Q$ is a prime ideal of $R$ then $f^{-1}[Q]$ is a prime ideal of $S$. 
It is not difficult to see that $f$ is unitary if and only if for all $Q \in \text{Spec} (R)$, $f^{-1}[Q] \in \text{Spec} (S)$.

For the reader's convenience, we define now the basic properties we consider.

 We say that $R$ satisfies LO (lying over) over $S$ if for all $P \in \text {Spec} (S)$ there exists $Q \in \text {Spec}(R)$ lying over $P$.


 We say that $R$
satisfies GD (going down) over $S$ if for any $P_{1} \subset
P_{2}$ in $\text{Spec} (S)$ and for every $Q_{2} \in \text{Spec} (R)$
lying over $P_{2}$, there exists $Q_{1} \subset Q_{2}$ in
$\text{Spec} (R)$ lying over $P_{1}$.

We say that $R$
satisfies GU (going up) over $S$ if for any $P_{1} \subset
P_{2}$ in $\text{Spec} (S)$ and for every $Q_{1} \in \text{Spec} (R)$
lying over $P_{1}$, there exists $Q_{1} \subset Q_{2}$ in
$\text{Spec} (R)$ lying over $P_{2}$.

We say that $R$
satisfies SGB (strong going between) over $S$ if for any $P_{1} \subset
P_{2} \subset
P_{3}$ in $\text{Spec} (S)$ and for every $Q_{1} \subset
Q_{3}$ in $\text{Spec} (R)$ such that $Q_{1}$ is lying over $P_{1}$ and
$Q_{3}$ is lying over $P_{3}$, there exists $Q_{1} \subset Q_{2} \subset Q_{3}$ in
$\text{Spec} (R)$ lying over $P_{2}$.

We say that $R$ satisfies INC (incomparability) over $S$ if
whenever $Q_{1} \subset Q_{2}$ in $\text{Spec}(R)$, we have $Q_{1} \cap S
 \subset Q_{2} \cap S $.

It is also customary to say that $f$ (the unitary homomorphism defined above) satisfies a certain property, instead of saying that $R$ satisfies this certain property over $S$.

\section{Lying over and incomparability}

In this section we present a necessary and sufficient condition for an algebra to satisfy LO over a valuation domain.
We also present a sufficient condition for an algebra over a valuation domain to satisfy INC over it.

\subsection{Lying over}

In [Sa1, Lemma 3.12] we proved that $O_{w}$ satisfies LO
over $O_{v}$ (for the assumptions assumed in [Sa1] see Remark \ref{remark on sa1}).
We shall now prove that the lying over property is
valid in a much more general case.


We start with the following remark, which is a generalization of [Sa1, Lemma 3.10].

\begin{rem} \label{x can be written as x=ar} Let $R$ be an algebra over $O_{v}$ and let $I$ be an ideal of $O_v$; then every $x \in IR$ can be written as $x=ar$ for $a \in I$, $r \in R$.
\end{rem}

\begin{proof} Let $x \in IR$. Then $x$ is of the form $\sum_ {i=1}^{n} a_{i}r_{i}$ for $a_{i}  \in I, \ r_{i} \in R$; so take $a=a_{i_{0}}$
with minimal $v$-value
and write $x=ar$ for an appropriate $r \in R$.
 \end{proof}

Let $R$ be an algebra over $O_{v}$ and let $w$ be the filter quasi-valuation induced by $(R,v)$.
We define the following properties:

(a) There exists an $F-$algebra $A$ such that $R$ is a subring of $A$ and

$R \cap F=O_{v}$;

(b) $R$ is a torsion-free algebra over $O_{v}$ and $(R^{\times} \cap O_{v}) \subseteq O_{v}^{\times}$;

(c) $w(a \cdot 1_R)=v(a)$ for all $a \in O_{v}$;

(d)  $R$ is faithful and $((R_P)^{\times} \cap (O_{v})_{P}) \subseteq ((O_{v})_{P})^{\times}$ for every $P \in \text{Spec}(O_{v})$;

(e) $R$ satisfies LO over $O_{v}$.

\begin{prop} \label{almost equivalent conditions LO} \textbf{LO.} The following implications hold:

(a) $\Leftrightarrow$ (b) $\Rightarrow$ (c) $\Rightarrow$ (d) $\Leftrightarrow$ (e).

\end{prop}

\begin{proof} (a) $\Rightarrow$ (b). It is clear that $R$ is a torsion-free $O_{v}-$algebra. Let $a \in R^{\times} \cap O_{v}$ and assume to the contrary that $a \notin O_{v}^{\times}$; then $a^{-1} \in R \cap F \setminus O_{v}$, a contradiction.

(b) $\Rightarrow$ (a). Define the $F-$ algebra $A=R \otimes_{O_{v}} F$. We identify $F$ as $1 \otimes_{O_{v}} F$ and, since $R$ is torsion-free over $O_{v}$, $R$ can be identified as $R \otimes_{O_{v}} 1$. It is obvious that $O_{v} \subseteq R \cap F$. Let $a \in R \cap F$
and assume to the contrary that $a \notin O_{v}$; then $a^{-1} \in O_{v}$ and thus $a^{-1} \in R$. So, $a^{-1} \in R^{\times} \cap O_{v} \setminus O_{v}^{\times}$, a contradiction.

(b) $\Rightarrow$ (c). By [Sa1, Remark 9.23], $w(a \cdot 1_R) \geq v(a)$. Assume to the contrary that there exists
$\beta \in w(a \cdot 1_R)$ such that $\beta > v(a)$ and let $b \in O_{v}$ with $v(b)=\beta$. Then $a \cdot 1_R=br$  for some $r \in R$. Now, $v(b)>v(a)$ and thus $a^{-1}b \in O_{v}$; therefore $a \cdot 1_R=a(a^{-1}b)r$. Since $R$ is torsion-free over $O_{v}$, one can cancel $a$ and get
$1_R=a^{-1}br$. Hence $a^{-1}b \in R^{\times} \cap O_{v} \setminus O_{v}^{\times}$, a contradiction.

(c) $\Rightarrow$ (d). If $R$ is not faithful then there exists a nonzero $a \in O_{v}$ such that $a \cdot 1_R=0$; but then $\infty=w(a \cdot 1_R)>v(a)$. Now, let $P \in \text{Spec}(O_{v})$
and assume to the contrary that $(R_P)^{\times} \cap (O_{v})_{P} \nsubseteq ((O_{v})_{P})^{\times}$. Let $p \in (R_P)^{\times} \cap (O_{v})_{P} \setminus ((O_{v})_{P})^{\times}$ (note that an element of $(O_{v})_{P} \setminus ((O_{v})_{P})^{\times}$ must be in $P$). Then there exists $rs^{-1} \in R_{P}$ ($r \in R, \ s \in O_{v} \setminus P$) such that $1_{R_P}=prs^{-1}$. So, there exists $s' \in O_{v} \setminus P$ such that $s'(pr-s \cdot 1_R)=0$; i.e., $s'pr=s's \cdot 1_R$. In particular, $w(s'pr)=w(s's \cdot 1_R)$. We prove that this is impossible. Indeed, by assumption $w(s's \cdot 1_R)=v(s's)$. Now, since
$s's \in O_{v} \setminus P$, $v(s's) < v(p)$. Also, by assumption, and since $w$ is the filter quasi-valuation induced by $(R,v)$, we have $$w(s'pr) \geq w(s'p \cdot 1_R) +w(r) =v(s'p)+w(r) \geq v(p),$$ a contradiction.

(d) $\Rightarrow$ (e). We start by proving that $PR \cap O_{v}=P$ for any $P \in \text{Spec}(O_{v})$. Let $P \in \text{Spec}(O_{v})$; it is obvious that $P \subseteq PR \cap O_{v}$. Assume to the contrary that $PR \cap O_{v} \supset P$
and let $b \in (PR \cap O_{v}) \setminus P$. Then there exist $p \in P$ and $r \in R$ such that $b \cdot 1_R = pr$ (note that every element of $PR$
can be written in this form, by Remark \ref{x can be written as x=ar}). We pass to $R_{P}$ over $(O_{v})_{P}$. By assumption $R$ is faithful over $O_{v}$ and thus $(b \cdot 1_R)(1_{O_{v}})^{-1}$ is nonzero in $R_{P}$. So,
$b \cdot 1_{R_{P}} = pr$ is nonzero in $R_{P}$; since $b^{-1} \in (O_{v})_{P}$, we get $1_{R_{P}} = b^{-1}pr$ in $R_{P}$. I.e., $b^{-1}p \in ((R_P)^{\times} \cap (O_{v})_{P}) \setminus ((O_{v})_{P})^{\times}$, a contradiction.

Now, define $ \mathcal Z=\{ I
\vartriangleleft R \mid I \cap O_{v}=P\}$. By the previous paragraph, $\mathcal Z$ is not empty; finally, a standard Zorn's Lemma argument
finishes the proof.

(e) $\Rightarrow$ (d). If $R$ is not faithful over $O_{v}$ then there exists a nonzero $a \in O_{v}$ such that $a \cdot 1_R=0$. In this case,
it is clear that there is no ideal in $R$ lying over the prime ideal $\{ 0 \} $ in $O_{v}$.
If there exists $P \in \text{Spec}(O_{v})$
 such that $(R_P)^{\times} \cap (O_{v})_{P} \nsubseteq ((O_{v})_{P})^{\times}$, we take $p \in (R_P)^{\times} \cap (O_{v})_{P} \setminus ((O_{v})_{P})^{\times}$; then $PR=R$ and there is no prime ideal in $R$ lying over $P$.

\end{proof}

We note that property (d) is strictly weaker than property (b). For example, let $v$ denote a $p-$adic valuation
on $\Bbb Q$ and let $\Bbb Z_{p}$ denote the corresponding valuation domain. Let $\Bbb Z_{p}[x]$ denote the polynomial ring over $\Bbb Z_{p}$ and let $R=\Bbb Z_{p}[x]/<px>$. Clearly, $R$ is not torsion-free over $\Bbb Z_{p}$ and it can be easily seen that $R$ satisfies property (d) (or (e)).





\subsection{Incomparability}

In this subsection we show that if $R$ is a torsion-free $O_{v}-$algebra and
$[R \otimes_{O_{v}}F:F] < \infty $ then $R$ satisfies INC over $O_{v}$. We deduce that
given $P \in \text {Spec}(O_{v})$, the number of prime ideals of $R$ lying over $P$ does not exceed
$[R \otimes_{O_{v}}F:F]$. Finally, we obtain a bound on the size of $\text {Spec}(R)$.

   Let $A$ be a ring, let $U$ be an $A-$module, and let $T \subseteq U$; we denote
$sp( T)=\{ \sum_{i=1}^{k} \alpha_{i}t_{i} \mid \alpha_{i} \in
A, t_{i} \in T, k \in \Bbb N\}$. We say that $T$ is
$A-$linearly independent if for any subset $\{t_{i}\}_{1 \leq i \leq k}
\subseteq T$, $\sum_{i=1}^{k} \alpha_{i}t_{i}=0$ (for $\alpha_{i}
\in A$) implies $\alpha_{i}=0$ for every $1 \leq i \leq k$.
We say that $T$ is a set of generators of
$U$ over $A$ (or that $T$ generates $U$ over $A$) if $U=sp(T)$.


Let $R$ be a torsion-free algebra over a commutative ring $S$ and consider the $F-$algebra $R \otimes_{S}F$. A subset $\{ r_i \}_{i \in I}$ of $R$
is $S-$linearly independent iff  $\{ r_i \otimes 1\}_{i \in I}$ is linearly independent over $F$.
In particular, every $S-$linearly independent set of elements of $R$
is finite iff $R \otimes_{S}F$ is a finite dimensional $F-$algebra.


The following lemma was proved in [Sa1] in a less general
form (see [Sa1, Lemma 2.5]). We shall prove it here for the reader's convenience.


\begin{lem} \label{R torsion-free implies residue degree less} Let $R$ be a torsion-free algebra over $O_{v}$. Then $$[R/I_{v}R:O_{v}/I_{v}]\leq
[R \otimes_{O_{v}}F:F].$$
\end{lem}

\begin{proof} Let $\{ \overline{r_{i}} \}_{i \in I} \subseteq R/I_{v}R$
be linearly independent over $O_{v}/I_{v}$. Let
$\{r_{i} \}_{i \in I} \subseteq R$ be a set of representatives. We prove that
$\{r_{i} \}_{i \in I}$ is $O_{v}-$linearly independent. Assume the
contrary and write $\sum_{i=1}^m
\alpha_{i}r_{i}=0$ where $\{ r_1, r_2, ..., r_m\} \subseteq \{r_{i} \}_{i \in I}$ and $\alpha_{i} \in O_{v}$ are not all zero.
Let $\alpha_{i_{0}}$ denote an element with minimal $v-$value; in particular, $\alpha_{i_{0}} \neq 0$. Then
one may write $\alpha_{i_{0}}(\sum_{i=1}^m
\alpha_{i_{0}}^{-1}\alpha_{i}r_{i})=0$. Now, since $R$ is torsion-free over $O_{v}$,  we have $\sum_{i=1}^m
\alpha_{i_{0}}^{-1}\alpha_{i}r_{i}=0$; specifically, $\sum_{i \neq i_{0}}
\alpha_{i_{0}}^{-1}\alpha_{i}r_{i}+r_{i_{0}}=0$. So, we have
$\sum_{i \neq i_{0}} \overline{\alpha_{i_{0}}^{-1}\alpha_{i}} \overline{r_{i}}+\overline{r_{i_{0}}}= \overline{0}$. This contradicts the
linear independence of $\{ \overline{r_{i}} \}_{i \in I}$.

\end{proof}


\begin{cor} \label{finite residue} Let $R$ be a torsion-free algebra over $O_{v}$. Let $I$ be an ideal of $ R$
lying over $I_{v}$; then $[R/I:O_{v}/I_{v}]\leq [R \otimes_{O_{v}}F:F]$.
\end{cor}

\begin{proof} Use the natural epimorphism $R/I_{v}R \twoheadrightarrow
R/I$ and Lemma \ref{R torsion-free implies residue degree less}. \end{proof}



\begin{lem} \label{INC for primes over Iv} 
Let $R$ be a torsion-free algebra over $O_{v}$ such that $[R \otimes_{O_{v}}F:F] < \infty $. If $Q_{1} \subseteq
Q_{2}$ are prime ideals of $R$ lying over $I_{v}$ then
$Q_{1}=Q_{2}$. In other words, every prime ideal of $R$ lying over $I_{v}$ is maximal.
\end{lem}

\begin{proof} We pass to $R /Q_{1}$ over $O_{v}/I_{v}$.
$R/Q_{1}$ is a prime algebra which is finite dimensional over the
field $O_{v}/I_{v}$, by Corollary \ref{finite residue}. Thus $R/Q_{1}$
is simple and $Q_{1}=Q_{2}$.
\end{proof}


\begin{thm} \label{INC} \textbf{INC.} Let $R$ be a torsion-free algebra over $O_{v}$ such that $[R \otimes_{O_{v}}F:F] < \infty $.
Then $R$ satisfies INC over $O_{v}$.
\end{thm}

\begin{proof} Let $Q_{1} \subseteq
Q_{2}$ be prime ideals of $R$ lying over $P \in \text {Spec}(O_{v})$. We
localize $O_{v}$ and $R$ at $P$ and consider $R_{P} $ as an algebra over the valuation domain $(O_v)_{P}$.
Note that $R_{P} $ is torsion-free over $(O_v)_{P}$ and $R_{P} \otimes_{O_{v}}F$ is a finite dimensional $F-$ algebra.
Now, $(Q_{1})_{P} \subseteq (Q_{2})_{P}$ are prime ideals of $R_{P}$ lying over the
valuation ideal $P=P_{P}$ of $(O_{v})_{P}$; thus by Lemma \ref{INC
for primes over Iv} $(Q_{1})_{P} = (Q_{2})_{P}$. Therefore
$Q_{1}=Q_{2}$.

\end{proof}

\begin{cor} \label{QP is a maximal ideal of RP} Notation and assumptions as in Theorem \ref{INC}. Let $P \in \text {Spec}(O_{v})$ and let $Q \in \text {Spec}(R)$ lying over $P$;
then $Q_{P}$ is a maximal ideal of $R_{P}$.\end{cor}

We note that the assumptions in Theorem \ref{INC} are necessary and one cannot omit any of them, as demonstrated in the following easy examples:


\begin{ex} Let $O_{v}$ be a valuation domain and let $R=O_{v}[x]$ denote the polynomial ring over $O_{v}$.
Then $<0> \subset <x>$ are prime ideals of $R$ lying over $\{ 0\} \in \text {Spec}(O_{v})$.\end{ex}

\begin{ex} Let $p \in \Bbb N$ be a prime number and consider the $p-$adic valuation on $\Bbb Q$. Let $\Bbb Z_{p}$ denote the
valuation domain and let $$R=\Bbb Z_{p}[\sqrt{p},\sqrt[3]{p}]/<p\sqrt{p}>.$$ It is clear that $R$ is torsion over $\Bbb Z_{p}$ and
every $\Bbb Z_{p}-$linearly independent set of elements of $R$ is finite. Finally, the chain
$\sqrt{p}R \subset \sqrt{p}R+\sqrt[3]{p}R$ is a chain of prime ideals of $R$ lying over the prime ideal $p \Bbb Z_{p}$.\end{ex}

Let $P$ be a prime ideal of $O_{v}$; we denote by $\mathcal{Q}_{P}$ the set of all prime ideals of $R$ that are
lying over $P$. Namely, $$\mathcal{Q}_{P}= \{Q \in \text {Spec}(R) | Q \cap O_{v}=P \} .$$

\begin{lem}  \label{set of prime ideals lying over
I_v has leq n elements} Notation and assumptions as in Theorem \ref{INC}.  
Then $$| \mathcal{Q}_{I_{v}} | \leq [R \otimes_{O_{v}}F:F].$$
\end{lem}

\begin{proof} First note that by Lemma \ref{INC for primes over Iv} every $Q \in \mathcal{Q}_{I_{v}}$ is
a maximal ideal of $R$. Let $\{ Q_{j}\}_{1 \leq j \leq m}
\subseteq \mathcal{Q}_{I_{v}}$. Then, by Lemma \ref{finite
residue}, $$[R \otimes_{O_{v}}F:F] \geq [R/ \cap_{j=1}^{m} Q_{j}:
O_{v}/I_{v}]=[\oplus_{j=1}^{m}
R/Q_{j}:O_{v}/I_{v}]=$$$$\sum_{j=1}^{m}[R/Q_{j}:O_{v}/I_{v}] \geq
m.$$
\end{proof}

\begin{lem} Notation and assumptions as in Theorem \ref{INC} and let $P \in \text {Spec}(O_{v})$. 
Then $| \mathcal{Q}_{P} | \leq [R \otimes_{O_{v}}F:F]$.
\end{lem}

\begin{proof} By Corollary \ref{QP is a maximal ideal of RP}, for every $Q \in \mathcal{Q}_{P}$, $Q_{P}$ is
a maximal ideal of $R_{P}$. Obviously $|\{ Q_{P}\}_{Q \in \mathcal{Q}_{P}}|=|\mathcal{Q}_{P}|$; thus by Lemma \ref{set of prime ideals lying over
I_v has leq n elements} $$| \mathcal{Q}_{P} | \leq [R_{P} \otimes_{O_{v}}F:F]
=[R \otimes_{O_{v}}F:F].$$
\end{proof}

\begin{prop} \label{specR is bounded from above} Notation and assumptions as in Theorem \ref{INC}. Then $$ |\text {Spec}(R)| \leq [R \otimes_{O_{v}}F:F] \cdot
|k-dimO_{v}|.$$

\end{prop}

\begin{proof} 
It is clear that $\text {Spec}(R)=\cup_{P \in \text {Spec}(O_{v})} \mathcal{Q}_{P}$. Now, by the previous lemma, $|\mathcal{Q}_{P}|
\leq [R \otimes_{O_{v}}F:F]$ for every $P \in \text {Spec}(O_{v})$.\end{proof}

We now state a generalization of Theorem 5.21, which was proven in [Sa1]
regarding the size of the prime spectrum of a \qv ring.
In [Sa1, Section 5] 
we also assumed that $M_{w}$ is a
torsion group over $\Gamma_{v}$ (see the basic notation in Remark \ref{remark on sa1}). We obtained an upper and a lower bound on the
size of the prime spectrum of the \qv ring. We shall now see that one can
obtain similar results for any torsion-free $O_{v}-$algebra $R$ such that $(R^{\times} \cap O_{v}) \subseteq O_{v}^{\times}$
and $[R \otimes_{O_{v}}F:F]< \infty$.

\begin{thm} \label{bound specR} Notation and assumptions as in Theorem \ref{INC} and assume in addition that
$(R^{\times} \cap O_{v}) \subseteq O_{v}^{\times}$. Then $$|k-dimO_{v}| \leq |\text {Spec}(R)| \leq [R \otimes_{O_{v}}F:F] \cdot
|k-dimO_{v}|.$$

\end{thm}

\begin{proof} The first inequality is by Proposition \ref{almost equivalent conditions LO}.
The second inequality is by Proposition \ref{specR is bounded from above}.
\end{proof}

Note that the assumptions in Theorem \ref{INC} do not imply that $R$ satisfies LO over $O_{v}$; thus we need to assume the additional assumption in
Theorem \ref{bound specR}. As an easy example, take $R$ to be any subring of $F$ strictly containing $O_{v}$; then
$R$ is a valuation domain of $F$, satisfies the assumptions of Theorem \ref{INC} and clearly does not satisfy LO over $O_{v}$. Moreover, the lower bound given in Theorem \ref{bound specR} is not valid.

We deduce from Proposition \ref{specR is bounded from above},

\begin{cor} \label{If k-dimOv<infty then the prime spectrum of R is finite}
Notation and assumptions as in Theorem \ref{INC}. If $k-dimO_{v}< \infty$ then the prime spectrum of $R$ is finite;
in particular, $R$ has a finite number of maximal ideals.
\end{cor}

Let $R$ be a torsion-free algebra over $O_{v}$ with $[R \otimes_{O_{v}}F:F] < \infty $. We concluded in
Corollary \ref{If k-dimOv<infty then the prime spectrum of R is finite} that $k-dimO_{v}< \infty$ implies that the set of
maximal ideals of $R$ is finite. However, the
other direction is not valid. For example, let $O_{v}
$ be a valuation domain with $k-dimO_{v}= \infty$ and let $R=M_{n}(O_{v})$; then $R$ has only one maximal ideal.
We will show in a subsequent paper
that when $E/F$ is a finite dimensional field extension and $R$ is a subring of $E$
lying over $O_{v}$, then the set of maximal
ideals of $R$ is always finite (even when $k-dimO_{v}=
\infty$).


\section{Going down and strong going between}


In this section we prove that a torsion-free algebra satisfies GD over a valuation domain.
We deduce that an algebra over a commutative valuation ring satisfies SGB over it.
Finally, we conclude that a torsion-free algebra satisfies GGD (generalized going doan) over a valuation domain.

Using the filter quasi-valuation, one can simplify some of the
proofs and generalize some of the theorems proven in [Sa1] to the non-commuative
case.

We note that one can generalize [Sa1, Lemma 4.11] to the case in which $R$ is an algebra over $O_{v}$.
However, we do not need this generalization for our needs. So, the following lemma is a special case of this generalization,
easily proven for the filter quasi-valuation.

\begin{lem} \label{b outside of Q has value} Let $R$ be an algebra over $O_{v}$. Let $w$ denote
the filter quasi-valuation induced by $(R,v)$.
Let $P \in \text{Spec}(O_{v})$, $Q \in \text {Spec}(R)$ such that $Q \cap O_{v} =
P$. Let $H$ denote the isolated subgroup corresponding to $P$ and
let $b \in R \setminus Q$. Then $w(b) \leq H ^{+}$.
\end{lem}

\begin{proof} Assume to the
contrary that $w(b)
> H ^{+}$; then there exists $\alpha \in \Gamma_{v}$
such that $w(b) \geq \alpha > H ^{+}$.
Let $a \in O_{v}$ such that $v(a)=\alpha$; since $\alpha > H ^{+}$, we get $a \in P$. Thus $b \in
aR \subseteq PR \subseteq Q$, a contradiction.

\end{proof}

\





In [Sa1, Lemma 4.12] we proved that $O_{w}$ satisfies GD over $O_{v}$ (for the assumptions made in [Sa1] see Remark \ref{remark on sa1}).
We mentioned there that one can generalize this fact for the case in which $E$ is a finite dimensional $F$-algebra.
We shall now present a much more general result.

\begin{thm} \label{R torsion-free over Ov satisfies GD}\textbf{GD.} Let $ R$ be a torsion-free algebra over $O_{v}$. Then
$R$ satisfies GD over $O_{v}$.
\end{thm}

\begin{proof} Let $ w$ denote the filter \qv on $R$
induced by $(R,v)$. 
Let $P_{1} \subset P_{2} \in \text {Spec}(O_{v})$ and let
$Q_{2} $ be a prime ideal of $R$ lying over $P_{2}$. We need to prove that there exists $Q_{1} \in \text {Spec}(R)$
lying over $P_{1}$ such that $Q_{1} \subset Q_{2}$. We
denote $S_{1} = O_{v} \setminus P_{1}$, $S_{2} = R \setminus
Q_{2}$ and $$S = \{ s_{1} s_{2} | s_{1} \in S_{1}, s_{2} \in S_{2}
\}.$$

Note that $S_{2}$ is an $m$-system in $R$ and thus $S$ is an
$m$-system in $R$. We shall prove $S \cap P_{1}R = \emptyset$ and
then every ideal $Q_{1}$ which contains $P_{1}R$, maximal with
respect to $S \cap Q_{1} = \emptyset$, is prime. Note that such
$Q_{1}$ satisfies the required properties: $Q_{1}$ is lying over $P_{1}$ and $Q_{1} \subset Q_{2}$.

Let $H_{i} \leq \Gamma_{v}$ $(i=1,2)$ be the isolated subgroups
corresponding to $P_{i}$. Let $x \in P_{1}R$ and write, by Remark \ref{x can be written as x=ar}, $x = pr$ where
$p \in P_{1}$ and $r \in R$. Then
$$w(x) = w(pr)  \geq
w(p \cdot 1_{R}) \geq v(p)  \notin H_{1}.$$
The first inequality is valid since $w$ is the filter \qv induced by $(R,v)$ (note that $w(r) \geq 0$ for all $r \in R$, by [Sa1, Remark 9.18]).
The second inequality is valid by [Sa1, Remark 9.23].
Viewing $v(p)$ inside $\mathcal M (\Gamma_{v})$ we get $$w(x)
\geq v(p) > H_{1}^{+}.$$

Now, let $y \in S$, and write $y=ab$ for $a \in S_{1}$, $b \in S_{2}$.
By Lemma \ref{b outside of Q has value}, $$w(b) \leq H_{2} ^{+}.$$ As
for $a$, $a \in S_{1}$ and thus $v(a) \in H_{1}$. Viewing $v(a)$
inside $\mathcal M (\Gamma_{v})$ we get $v(a) < H_{1} ^{+}$. Since $R$ is torsion-free over $O_{v}$ we have, by Lemma \ref{Sa1, Lemma 9.25},
$$w(ab) = v(a) + w(b) \leq H_{1}^{+} + H_{2} ^{+} =
H_{1}^{+}.$$ So, we have proven that for every $s \in
S, $ $ w(s) \leq H_{1}^{+} $ and for every $x \in
P_{1}R$, $w(x)
>H_{1} ^{+}$. Therefore $$ S \cap P_{1}R = \emptyset.$$

 \end{proof}

It is easy to see that if $R$ is not torsion-free over $O_{v}$ then $R$ does not necessarily satisfy GD over $O_{v}$.
In fact, for any nontrivial valuation domain $O_{v}$ there exists an algebra $R$ that does not satisfy GD over it. Indeed,
take any algebra $R$ that is not faithful over $O_{v}$ (for example, take $R=O_{v}/I$ where $I$ is any nonzero proper ideal of $O_{v}$); then $R$ does not satisfy GD over $O_{v}$, since
there is no prime ideal of $R$ lying over $\{ 0 \}$.

For any ring $R$, we denote by K-dim$R$ the classical Krull
dimension of $R$, by which we mean the maximal length of the
chains of prime ideals of $R$.

\begin{cor} Let $R$ be a torsion-free $O_{v}$-algebra such that $(R^{\times} \cap O_{v}) \subseteq O_{v}^{\times}$ and $R \otimes_{O_{v}}F$ is a finite dimensional $F-$algebra. Then
K-dim$R =$K-dim$O_{v}$.\end{cor}

 \begin{proof} By Theorem \ref{INC}, $R$
satisfies INC over $O_{v}$; thus, we get K-dim$R  \leq
$K-dim$O_{v}$. By Proposition \ref{almost equivalent conditions LO}, $R$ satisfies LO over $O_{v}$
and by Theorem \ref{R torsion-free over Ov satisfies GD}, $R$ satisfies GD over $O_{v}$; thus
K-dim$R \geq $K-dim$O_{v}$. \end{proof}









Using Theorem \ref{R torsion-free over Ov satisfies GD}, we can now deduce a useful corollary regarding commutative valuation rings that are not necessarily integral domains.
Recall that a commutative valuation ring $S$ is a commutative ring such that the set of ideals of $S$ is totally ordered with respect to containment.

\begin{cor} \label{R torsion-free over S commutative valuation satisfies GD} Let $S$ be a commutative valuation ring and let $ R$
be an algebra over $S$.
Then $R$ satisfies GD over $S$ iff every minimal prime ideal of $R$ is lying over the minimal prime ideal of $S$.
\end{cor}

\begin{proof} $(\Leftarrow)$ Let $P_{1} \subset P_{2} \in \text {Spec}(S)$ and let
$Q_{2} $ be a prime ideal of $R$ lying over $P_{2}$. Let $Q \subseteq Q_{2}$ be a minimal prime ideal of $R$.
By assumption, $Q$ is lying over
the minimal prime ideal of $S$, say $P$; in particular, $Q \neq Q_{2}$ since $P_{2}$ is not a minimal prime ideal of $S$.
Now, it is clear that $R/Q$ is a torsion-free algebra over the valuation domain $S/P$. Thus, by Theorem \ref{R torsion-free over Ov satisfies GD},
$R/Q$ satisfies GD over $S/P$; hence, there exists a prime ideal $Q_{1}/Q \subset Q_{2}/Q$ such that $Q_{1}/Q $ is lying over $P_{1}/P$.
Therefore, there exists a prime ideal $Q_{1} \subset Q_{2}$ of $R$ lying over $P_{1}$.

$(\Rightarrow)$ It is clear.

 \end{proof}

Using Corollary \ref{R torsion-free over S commutative valuation satisfies GD}, we can now present an example demonstrating the fact that Theorem \ref{R torsion-free over Ov satisfies GD} cannot be reversed. That is, $R$ being a torsion-free algebra over $O_{v}$ is a sufficient condition, but not a necessary condition, for $R$ to satisfy GD over $O_{v}$.

\begin{ex} Let $O_{v}$ be a valuation domain and let $a \in O_{v}$ with $v(a)>0$.
Let $O_{v}[x]$ be the polynomial ring over $O_{v}$ and
let $R=O_{v}[x]/<ax,x^2>$. Clearly, $R$ is not torsion-free over $O_{v}$ and it
is not difficult to see that every minimal prime ideal of $R$ (namely $xR$) is lying over $\{ 0 \} $.

\end{ex}



Let $S$ and $R$ be commutative rings and let $f: S \rightarrow R$ be a unitary ring homomorphism. In [Pi, Propositions 5.2 and 5.7],
Picavet presented some equivalent conditions for $f$ to satisfy SGB. In particular, he showed that $f$ satisfies SGB
iff for all $Q \in \text {Spec}(R)$ the induced map $\overline{f}: S/f^{-1}[Q] \rightarrow R/Q$ satisfies GD.
It is not difficult to see that this property is valid in a more general case.
We give the following lemma without a proof.


\begin{lem} \label{SGB iff GD for all prime hom images} Let $S$ and $R$ be rings (not necessarily commutative) and let
$f: S \rightarrow R$ be a homomorphism (not necessarily unitary) such that $f[S] \subseteq Z(R)$. Then
$f$ satisfies SGB iff for all $Q \in \text {Spec}(R)$ the induced map $\overline{f}: S/f^{-1}[Q] \rightarrow R/Q$ satisfies GD.
\end{lem}


We say that a homomorphism $f: S \rightarrow R$ is {\it torsion-free} if for all $0 \neq s \in S$ and $0 \neq r \in R$, 
one has $f(s)r \neq 0$. As a side note we mention that if $f \neq 0$ is torsion-free then it is unitary.
If $f[S] \subseteq Z(R)$ and $S$ is commutative then viewing $R$ as an $S$-algebra (in the natural way: $s \cdot r \doteqdot f(s)  r$),
it is obvious that $f$ is torsion-free iff $R$ is torsion-free over $S$.

\begin{thm} \label{SGB} \textbf{SGB.} Let $S$ be a commutative valuation ring (not necessarily an integral domain) and let $ R$ be a ring.
Let $f: S \rightarrow R$ be a
homomorphism (not necessarily unitary) such that $f[S] \subseteq Z(R)$. Then
$f$ satisfies SGB. In particular, every $S$-algebra satisfies SGB over $S$.
\end{thm}

\begin{proof} Let $Q \in \text {Spec}(R)$. It is easy to see that the induced map $\overline{f}: S/f^{-1}[Q] \rightarrow R/Q$
is torsion-free. Thus, $R/Q$ is a torsion-free algebra over the valuation domain $S/ f^{-1}[Q]$.
(Note that we may have $f^{-1}[Q]=S$, since $f$ is not necessarily unitary, or $S/ f^{-1}[Q]$ a field;
in these cases it is clear that $R/Q$ satisfies GD over $S/ f^{-1}[Q]$, in a trivial way.)
 Hence, by Theorem \ref{R torsion-free over Ov satisfies GD}, $R/Q$ satisfies GD over $S/ f^{-1}[Q]$; namely, $\overline{f}$
 satisfies GD. The result now follows from Lemma \ref{SGB iff GD for all prime hom images}.
 \end{proof}

Recall from [Sa3] that an $S-$algebra $R$ is said to satisfy GGD (generalized going down) over $S$ if, for every
chain of prime ideals $\mathcal D$ of $S$ with a final element $P_0$ and $Q_0$ a
prime ideal of $R$ lying over $P_0$, there exists a chain of prime ideals $\mathcal C$ of $R$ covering $\mathcal D$
(namely, for every $P \in \mathcal D$ there exists $Q \in \mathcal C$ lying over $P$), whose final element is $Q_{0}$.
We deduced in [Sa3, Corollary 2.8] that if $R$ satisfies GD and SGB over $S$, then $R$ satisfies GGD over $S$. The following corollary is now obvious.

\begin{cor} Let $R$ be a torsion-free $O_{v}-$algebra. Then $R$ satisfies GGD over $O_{v}$. If in addition
$(R^{\times} \cap O_{v}) \subseteq O_{v}^{\times}$ (thus $R$ satisfies LO over $O_{v}$, by Proposition \ref{almost equivalent conditions LO}),
then for every chain in $\text{Spec}(O_{v})$ there exists a chain in $\text {Spec}(R)$ covering it.

\end{cor}

\section{Going up}

In this section we present some sufficient conditions for a \qv ring to satisfy GU over $O_{v}$.
Moreover, we show that any finitely generated algebra over a commutative local ring satisfies GU over it.
Finally, we obtain a connection between finitely generated algebras over valuation domains and a special kind of quasi-valuations.

In [Sa1, Section 5], in addition to the assumptions presented in Remark \ref{remark on sa1}, we
assumed that $M_{w}$ is a torsion group over $\Gamma_v$; we proved that $O_{w}$ satisfies GU over $O_{v}$.
In this section we shall generalize this fact.
The proof requires a totally different approach than the one in [Sa1].

\begin{defn} Let $R$ be an $O_{v}-$algebra and let $w$ be a \qv on $R$. $w$ is called a {\it $v-$quasi-valuation with respect to $O_{v}$} if
$w(cx)=v(c)+w(x)$ for all $c \in O_{v}$ and $x \in R$. Similarly, if $R$ is an $F-$algebra and
$w$ is a \qv on $R$ satisfying $w(\alpha x)=v(\alpha)+w(x)$ for all $\alpha \in F$ and $x \in R$, we say
that $w$ is a {\it $v-$quasi-valuation with respect to $F$}.
\end{defn}

When the base ring is understood we suppress the words "with respect to ...".

\begin{defn} Let $R$ be an $O_{v}-$algebra and let $w$ be a \qv on $R$. We say that $w$ {\it extends} $v$ on $O_{v}$ if
$w(c \cdot 1_R)=v(c)$ for all $c \in O_{v}$.\end{defn}

We note that the previous definition is a generalization of the one given in [Sa1, Section 1].

\begin{rem} \label{w(1)=0 implies w extends v} Let $R$ be an $O_{v}-$algebra and let $w$ be a $v-$\qv on $R$. If $w(1_R)=0$ then $w$ extends $v$
on $O_{v}$; in particular, $\Gamma_{v} \subseteq M_{w}$ and $R$ must be faithful over $O_{v}$ (see also Proposition
\ref{almost equivalent conditions LO}). So, $O_{v}$ embeds in $R$ and for $c \in O_{v}$, we often write $c \in R$ instead of $c \cdot 1_{R} \in R$.
\end{rem}

In [Sa1, Section 1] we defined the notion of a stable element with respect to a \qv on a commutative ring.
In order to present the next results in a clearer way, we generalize this notion in a very natural way.

\begin{defn} Let $w$ be a quasi-valuation on a
ring $R$. An element $c \in R$ is called {\it left stable}
with respect to $w$ if $$w(cr)=w(c)+w(r)$$ for every $r \in R$.
Analogously, one defines the notion {\it right stable}.

\end{defn}

Note that if $R$ is an $O_{v}-$algebra and $w$ is a $v-$quasi-valuation (with respect to $O_{v}$) satisfying $w(1_R)=0$, then every element of
$O_{v}$ is left stable with respect to $w$.

The following lemma is a generalization of Lemma 1.6 in [Sa1], with a similar proof.

\begin{lem} \label{left stable iff} Let $w$ be a quasi-valuation on a
ring $R$ such that 
$w(1)=0$. Let $x$ be a right invertible element of $R$ and let $y$ denote its right inverse. The following implications hold:

$x$ is left stable $\Rightarrow$ $w(y)=-w(x)$ $\Rightarrow$ $x$ is right stable.\end{lem}

We note that the left/right dual of Lemma \ref{left stable iff} is also valid.

Before stating the next remark, we recall that for a right (resp. left) artinian ring $R$, every element of $R$ is either left
(resp. right) zero-divisor or right (resp. left) invertible.

\begin{rem} Let $R$ be a right (resp. left) artinian ring and let $w$ be a \qv on $R$ such that $w(x) \neq \infty$ for all nonzero $x \in R$. If $0 \neq y \in R$ is left (resp. right) stable
with respect to $w$, then $y$ is right (resp. left) invertible in $R$.

\end{rem}

\begin{proof} Let $0 \neq y \in R$ be a left stable element. We prove that $y$ is not a left zero-divisor. Write $yr=0$ for some $r \in
R$; then $$\infty=w(yr)=w(y)+w(r).$$ Thus $w(r)=\infty$ i.e.,
$r=0$. The dual case is similarly proven.

\end{proof}

\begin{lem} Let $R$ be an $O_{v}$-algebra and let
$w$ be a $v-$\qv on $R$. Let $P \in \text {Spec}(O_{v})$ and let $K$ be a right (resp. left) ideal of $O_{w}$ lying over $P$.
Let $H$ denote the isolated subgroup corresponding to $P$.
Let $k \in K$ be a right (resp. left) invertible element in $R$ and let $y$ denote its right (resp. left) inverse.
If $w(y)=-w(k)$ then $w(k)>\alpha$ for every $\alpha \in H$.
In particular, assuming that $R$ is right (resp. left) artinian, $w(1_R)=0$, and
$w(x) \neq \infty$ for all nonzero $x \in R$; if $k'  \in K$ is left (resp. right) stable, then $w(k')> \alpha $ for every
$\alpha \in H$.

\end{lem}

\begin{proof} 
Assume to the contrary that $w(k) \leq \alpha_0 $ for some
$\alpha_0 \in H$. Let $a \in O_{v}$ with $v(a)=\alpha_0$; then $a \notin P$.
Now, since $w$ is a $v-$\qv and $w(y)=-w(k)$, we have $w(ay)=v(a)+w(y) \geq 0$; thus, $$a \cdot 1_R=a(ky)=k(ay) \in KO_{w}=K,$$ a contradiction.
The second assertion is deduced by the previous Remark, Lemma \ref{left stable iff}, and the first part of the lemma (note that if $k'=0$ then
obviously $w(k')> \alpha $ for every
$\alpha \in H$). The dual case is similarly proven.


\end{proof}



The following lemma is of utmost importance for our study.

\begin{lem} \label{K_{0}+K_{1} is lying over P1} Let $R$ be an $O_{v}$-algebra and let
$w$ be a $v-$\qv on $R$ such that $w(1_R)=0$ and $w(x) \neq \infty$ for all nonzero $x \in R$. Let $P_{0} \subseteq
P_{1} \in \text {Spec}(O_{v})$ and let $H_i$ denote the isolated subgroups corresponding to $P_i$ ($i=1,2$).
Assume that $R$ is right (resp. left) artinian and
$M_{w}$ is cancellative.
If $K_{0}$ is a right ideal (resp. left) of $O_{w}$
lying over $P_{0}$ and $K_{1}$ is any subset of $R$ such that $w(k_{1})> \alpha_1$ for every $k_{1} \in K_{1}$ and $\alpha_1 \in H_{1}$, then $(K_{0}+K_{1}) \cap O_{v} \subseteq P_{1}$.

\end{lem}

\begin{proof} 
Assume to the contrary that $(K_{0}+K_{1})  \cap O_{v} \nsubseteq P_{1} $ and write 

\begin{equation}
k_{o}+k_{1}=b  \tag{I} \label{2}
\end{equation}

where $k_{o} \in K_0, \ k_{1} \in K_{1}, $ and $b
\in O_{v } \setminus P_{1}$. By assumption, $w(k_{1})>v(b)$. Therefore by [Sa1, Lemma 1.4],
$$w(k_{0})=w(b)=v(b) \in H_{1} \subseteq H_{0}.$$

Now, if $k_{0}$ is right invertible in $R$ and $y$ is its right inverse, then by the previous lemma,
$w(y) \neq -w(k_{0})=-v(b)$. 
Since $w$ is a quasi-valuation and by assumption $w(1)=0$, we get $ w(k_0)+w(y) \leq 0$; 
thus, $w(y)<-w(k_{0})$.
Multiplying equation
(\ref{2}) by $y$ from the right, we get
$$1+k_{1}y=by.$$ Now, since $w(by)=v(b)+w(y)<0$ and $w(1)=0$, we have, by [Sa1, Lemma 1.4],
$w(k_{1}y)=w(by)$. Therefore,
$$w(k_{1})+w(y) \leq w(k_{1}y)=w(by)= v(b)+w(y).$$
Now, cancel $w(y)$ from both sides and get $w(k_{1})
\leq v(b)$, a contradiction. Note that $w(y)<-w(k_{0})$, so we do not need here the assumption that $0 \in R$
is the only element whose value is infinity.

If $k_{0}$ is not right invertible in $R$ then it is a left zero-divisor.
Write $k_{0}s=0$ for some nonzero $s \in R$.
Now, multiplying equation (\ref{2}) by $s$ from the right, we get
$$k_{0}s+k_{1}s=bs.$$ Thus $k_{1}s=bs$ and $$w(k_{1})+w(s) \leq
w(k_{1}s)=w(bs)=v(b)+w(s).$$
Now, since $s$ is not zero, $w(s) \neq \infty$; i.e., $w(s) \in M_{w}$ and is thus reducible.
Finally, cancel $w(s)$ from both sides
and get $w(k_{1}) \leq v(b)$, a contradiction.
As above, the dual case is similarly proven.

\end{proof}

\begin{prop} \label{Q_{0}+P_{1}O_{w} is lying over P1} Let $R$ be an $O_{v}$-algebra and let
$w$ be a $v-$\qv on $R$ such that $w(1_R)=0$ and $w(x) \neq \infty$ for all nonzero $x \in R$.
Let $P_{0} \subseteq
P_{1} \in \text {Spec}(O_{v})$.
Assume that $R$ is right or left artinian and
$M_{w}$ is cancellative. If $I_{0}$ is an ideal of $O_{w}$
lying over $P_{0}$ then $I_{0}+P_{1}O_{w}$ is lying over $P_{1}$.

\end{prop}
\begin{proof}
By definition, for all $x \in O_{w}$, $w(x) \geq 0$. Since $w$ is a $v-$quasi-valuation, we have $w(p_{1}x)>\alpha_1$ for every $p_{1} \in P_{1}$,
$x \in O_{w}$, and $\alpha_1 \in H_{1}$ (note that every element of $P_{1}O_{w}$ can be written in this form, by Remark \ref{x can be written as x=ar}). Now, Take $K_{0}=I_{0}$ and $K_{1}=P_{1}O_{w}$ in Lemma \ref{K_{0}+K_{1} is lying over P1}.
\end{proof}

As a side note, we mention that in [Sa1, Lemma 5.3] we proved that $I_{w}$ is contained in
any maximal ideal of $O_{w}$ (we assumed that $M_{w}$ is a torsion group over
$\Gamma_{v}$, in addition to the basic assumptions presented in Remark \ref{remark on sa1}).
We deduced that every maximal ideal
of $O_{w}$ is lying over $I_{v}$. The following Proposition is a
generalization of this fact.


\begin{prop} Let $R$, $w$, and $M_w$ be as in Lemma \ref{K_{0}+K_{1} is lying over P1}.
Let $K$ be a maximal right (resp. left) ideal of $R$. Then $I_{w}
\subseteq K$; in particular, $K \cap O_{v}=I_{v}$.
\end{prop}

\begin{proof} Take $K_{0}=K$ and $K_{1}=I_{w}$ in Lemma \ref{K_{0}+K_{1} is lying over P1}.
Then the right (resp. left) ideal $K+I_{w}$ is lying over $I_{v}$; in particular, it is a proper right (resp. left) ideal of $O_{w}$.
By the maximality of $K$, $I_{w} \subseteq K$.

\end{proof}

We take a small pause here to prove a general basic lemma
in which the base ring considered is any commutative ring and not necessarily a valuation domain.

\begin{lem} \label{Q_0+P_1R lying over P_1 iff there exists Q1} Let $S$ be a commutative ring and let $R$ be an $S-$algebra.
Let $P_{0} \subseteq P_{1} \in \text {Spec}(S)$ and $I_{0} \vartriangleleft R$ lying over $P_{0}$.
Then $I_{0}+P_{1}R$ is lying over $P_{1}$ iff
there exists $I_{0} \subseteq Q_{1} \in
\text {Spec}(R)$ such that $Q_{1} \cap S=P_{1}$.
\end{lem}

\begin{proof} $(\Rightarrow)$ Define $$\mathcal
Z=\{ J \vartriangleleft R \mid I_{0}+P_{1}R \subseteq J, \
J \cap S=P_{1}\}.$$ By assumption,
$I_{0}+P_{1}R$ is an ideal of $R$ lying over $P_{1}$; thus
$\mathcal Z \neq \emptyset $. Now, $\mathcal Z$ with the partial
order of containment satisfies the conditions of Zorn's Lemma and
therefore there exists $I_{0}  \subseteq I_{0}+P_{1}R \subseteq Q_{1} \vartriangleleft R$
lying over $P_{1}$, maximal with respect to containment. It is
easily seen that $Q_{1} \in \text {Spec}(R)$.
$(\Leftarrow)$ It is clear.

 \end{proof}

We return now to our general discussion, to prove one of the main results of this section.

\begin{thm} \label{GU} \textbf{GU.} Let $R$ be an $O_{v}$-algebra and let
$w$ be a $v-$\qv on $R$ such that $w(1_R)=0$ and $w(x) \neq \infty$ for all nonzero $x \in R$. Assume that $R$ is right (or left) artinian and
$M_{w}$ is cancellative. Then $O_{w}$ satisfies GU over $O_{v}$.

\end{thm}

\begin{proof} Let $P_{0} \subseteq P_{1} \in \text {Spec}(O_{v})$
and let $Q_{0} \in \text {Spec} (O_{w})$ lying over $P_{0}$. By Proposition
\ref{Q_{0}+P_{1}O_{w} is lying over P1}, $Q_{0}+P_{1}O_{w}$ is lying over $P_{1}$. The theorem now follows from Lemma \ref{Q_0+P_1R lying over P_1 iff there exists Q1}.

\end{proof}

Note that in Theorem \ref{GU} $R$ does not need to be torsion-free over $O_{v}$.



\begin{cor} Let $A$ be a right (or left) artinian $F$-algebra and let
$w$ be a \qv on $A$ extending $v$ on $F$, satisfying $w(x) \neq \infty$ for all nonzero $x \in A$. If
$M_{w}$ is cancellative then $O_{w}$ satisfies GU over $O_{v}$.

\end{cor}

\begin{proof} By assumption, $w$ extends $v$ on $F$ and thus, by [Sa1, Lemma 1.6], $w$ is a $v-$\qv with respect to $F$. In particular,
$w$ is a $v-$\qv with respect to $O_{v}$. Clearly, $w(1_A)=0$; the corollary now follows form Theorem \ref{GU}.


\end{proof}

We recall that in [Sa1, Theorem 5.16] we proved that $O_{w}$ satisfies GU over $O_{v}$
(under the assumptions presented in Remark \ref{remark on sa1} and
the assumption that $M_{w}$ is a torsion group over $\Gamma_v$). It is now easy to deduce
this theorem from [Sa1, Lemma 2.8] (which states that $w(x) \neq \infty$ for all nonzero $x \in E$) and the previous corollary.

The quasi-valuation ring considered in Theorem \ref{GU} is, by definition, an $O_{v}-$subalgebra of $R$
(note that if $w$ is the filter \qv induced by $(R,v)$, then $O_{w}=R$).
We now expand Theorem \ref{GU}.
Namely, we prove the GU property in case the \qv on $R$
can be extended in a natural way (see Lemma \ref{[Sa1, Lemma 9.31]}) and the associated quasi-valuation ring is not necessarily an $O_{v}-$subalgebra of $R$.





Using Remark \ref{w(1)=0 implies w extends v} and Lemma \ref{[Sa1, Lemma 9.31]} (with the mentioned construction), we deduce,

\begin{lem} \label{natural extension} Let $R$ be a torsion-free $O_{v}-$algebra and let $w$ be a $v-$\qv on $R$ (with respect to $O_{v}$). If $w(1_R)=0$ then
there exists a $v-$\qv $W$ (with respect to $F$), on $R
\otimes_{O_{v}}F$, extending $w$ on $R$ and extending $v$ on $F$, with $M_W=M_w$.
\end{lem}

Recall that $W$ is defined by $W( r\otimes \frac{1}{b})=w(r)-v(b)$ for all $r \otimes
\frac{1}{b} \in R \otimes_{O_{v}}F$.

\begin{defn} Notation as in Lemma \ref{natural extension}. $W$ is called the {\it natural extension} of $w$. \end{defn}

We present here two examples demonstrating that the \qv ring associated to the natural extension
need not be an $O_{v}-$subalgebra of $R$:

\begin{ex}  Let $I$ be a proper nonzero ideal of $O_{v}$ and let \begin{equation*} R_1=
 \left(
\begin{array}{cc}
 O_{v} &  I  \\
 I^{-1}  & O_{v}  \\

\end{array} \right), \text{where  }  I^{-1}=\{ x \in F \mid xI \subseteq O_{v}\}.
\end{equation*}

Define \begin{equation*} w_1
 \left(
\begin{array}{cc}
 a &  b  \\
 c  & d  \\

\end{array} \right)= \min \{ v(a), v(b), v(c), v(d)\} \text{ for all } \left(
\begin{array}{cc}
 a &  b  \\
 c  & d  \\

\end{array} \right) \in R_1.
\end{equation*}

Note that $w_1$ is a $v-$\qv satisfying $w(1)=0$ and $w(x) \neq \infty$ for all nonzero $x \in R_1$. Let $W_1$ denote the natural extension of $w_1$. Then $$O_{W_1} \nsubseteq R_1 \text{ and } O_{W_1} \nsupseteq R_1.$$

Similarly, let

\begin{equation*} R_2= \left(
\begin{array}{cc}
 O_{v} & \{ 0 \} \\
\{ 0 \} & I  \\

\end{array} \right),\text{ and define } w_2 \text{ as above.}
\end{equation*}

Then  $R_2 \subset O_{W_2}$, where $W_2$ is the natural extension of $w_2$.



\end{ex}

\begin{cor} \label{If exist subalgebra such that and contains then GU} Let $R$ be a torsion-free $O_{v}$-algebra.
Let $w$ be a $v-$\qv on $R$ such that $w(1_R)=0$, $w(x) \neq \infty$ for all nonzero $x \in R$
and $M_w$ is cancellative.
Let $W$ denote the natural extension of $w$. If there exists an
$O_{v}-$subalgebra $R'$ of $R \otimes_{O_{v}} F$ such that $R'$ is left or right artinian and $O_{W}$ embeds in $R'$, then $O_{W}$
satisfies GU over $O_{v}$.


\end{cor}

\begin{proof} Consider the $v-$quasi-valuation $W|_{R'}$ and note that by Lemma \ref{natural extension}, the pair $(R',W|_{R'})$ satisfies the conditions of Theorem \ref{GU}.

\end{proof}

For a ring $T$, we denote by $\text{Max}(T)$ the set of all maximal ideals of $T$.
\begin{rem} \label{remark regarding INC and GU} Let $S$ be a commutative ring and let $R$ be an $S-$algebra. If $R$ satisfies GU over $S$,
then $Q \in \text{Max}(R) \Rightarrow Q \cap S \in \text{Max}(S)$. If $R$ satisfies INC over $S$,
then $Q \notin \text{Max}(R) \Rightarrow Q \cap S \notin \text{Max}(S)$.

\end{rem}

\begin{cor} \label{Max O_W is finite} Let $R$ be a torsion-free $O_{v}-$algebra such that $[R \otimes_{O_{v}}F:F] < \infty $.
Assume that there exists a $v-$\qv $w$ on $R$ such that $w(1_R)=0$, $w(x) \neq \infty$ for all nonzero $x \in R$, and $M_w$ is cancellative. Then $\text{Max}(O_{W})=\mathcal{Q}_{I_{v}}$,
where $W$ denotes the natural extension of $w$. In particular, $\text{Max}(O_{W})$ is finite.

\end{cor}

\begin{proof} By assumption, $R \otimes_{O_{v}}F$ is a finite dimensional $F-$algebra and is thus artinian. Hence, by Corollary \ref{If exist subalgebra such that and contains then GU},
$O_{W}$ satisfies GU over $O_{v}$. It is clear that $O_{W}$ satisfies the assumptions of Theorem \ref{INC}, and thus $O_{W}$ satisfies INC over $O_{v}$.
Therefore, by Remark \ref{remark regarding INC and GU}, $\text{Max}(O_{W})=\mathcal{Q}_{I_{v}}$.
Finally, by Lemma \ref{set of prime ideals lying over
I_v has leq n elements}, $| \mathcal{Q}_{I_{v}} | \leq [R \otimes_{O_{v}}F:F].$

\end{proof}

\begin{cor} \label{maximal chain covers specO_v} Notation and assumptions as in Corollary \ref{Max O_W is finite} and let $\mathcal C$ be a maximal chain (with respect to containment) in $\text{Spec}(O_{W})$.
Then the map $Q \rightarrow Q \cap O_{v}$ is a bijective order preserving correspondence from $\mathcal C$ to $\text{Spec}(O_{v})$.
In other words, every maximal chain in $\text{Spec}(O_{W})$ is lying over $\text{Spec}(O_{v})$ in a one-to-one correspondence.
\end{cor}

\begin{proof} By Theorem \ref{INC}, $O_{W}$ satisfies INC over $O_{v}$. By Corollary \ref{If exist subalgebra such that and contains then GU},
$O_{W}$ satisfies GU over $O_{v}$. By Theorem \ref{R torsion-free over Ov satisfies GD}, $O_{W}$ satisfies GD over $O_{v}$, and by Theorem \ref{SGB} $O_{W}$ satisfies SGB over $O_{v}$.
The result now follows from [Sa3, Corollary 1.15].

\end{proof}

We note that if we do not assume in Corollary \ref{maximal chain covers specO_v} that $[R \otimes_{O_{v}}F:F] < \infty $ but instead assume that
$R \otimes_{O_{v}}F$ is left or right artinian, then the map above is not necessarily
injective (since we do not necessarily have INC), but it is surjective, as proven in [Sa3, Theorem 1.14] (since we have GU, GD and SGB).
In other words, in this case every maximal chain in $\text{Spec}(O_{W})$ is a cover of $\text{Spec}(O_{v})$.







We shall now study the GU property from a different point of view, without considering a valuation nor a quasi-valuation.
We shall assume a finiteness property.


\begin{rem} \label{Nakayama} Let $S $ be a commutative local ring, with maximal ideal $N$.
Let $R$ be an $S$-algebra, finitely generated as an $S$-module. Then,
by [Re, Theorem 6.15] (which is an application of Nakayama's Lemma), $NR \subseteq J(R)$,
where $J(R)$ is the Jacobson radical of $R$; clearly $J(R) \subseteq K$ for any maximal
ideal $K$ of $R$.\end{rem}

\begin{lem} \label{R f.g implies R* cap S contained in S*} Let $S $ be a commutative local ring and let $R$ be an $S$-algebra,
finitely generated as an $S$-module. Then $(R^{\times} \cap S) \subseteq S^{\times}$.
\end{lem}

\begin{proof} Let $N$ denote the maximal ideal of $S$ and assume to the contrary that there exists
$a \in (R^{\times} \cap S) \setminus S^{\times}$.
Thus $a \in N$ and $NR=R$, contradicting Remark \ref{Nakayama}.

\end{proof}

\begin{prop} Let $S $ be a commutative local ring and let $R$ be an $S$-algebra. 
If $R$ is finitely
generated as an $S-$module then $R$ satisfies GU over $S$.

\end{prop}

\begin{proof} First note that for any $P \in \text {Spec}(S)$, $R_{P}$ is an
$S_{P}$-algebra, finitely generated as an $S_{P}$-module. Thus, it
is enough to show that for any maximal ideal $K$ of $R$, $K$ is
lying over $N$, where $N$ is the maximal ideal of $S$. By Remark \ref{Nakayama}, the proposition is proved.

\end{proof}






In particular, we deduce the following:

\begin{cor} \label{R f.g implies GU} Let $S $ be a commutative valuation ring and let $R$ be an $S$-algebra, finitely generated as an $S$-module. Then $R$ satisfies GU
over $S$.
\end{cor}

\begin{cor} Let $R$ be a torsion-free $O_{v}-$algebra which is finitely generated as an $O_{v}-$module;
then $\text{Max}(R)=\mathcal{Q}_{I_{v}}$.
In particular, $\text{Max}(R)$ is finite.

\end{cor}





The proof of the previous corollary is quite similar to the proof of Corollary \ref{Max O_W is finite}:
use Corollary \ref{R f.g implies GU} to prove GU (instead of Corollary \ref{If exist subalgebra such that and contains then GU}); and note that
$[R \otimes_{O_{v}}F:F] < \infty$.

\begin{cor} Let $R$ be a torsion-free $O_{v}-$algebra which is finitely generated as an $O_{v}-$module,
and let $\mathcal C$ be a maximal chain in $\text{Spec}(R)$.
Then the map $Q \rightarrow Q \cap O_{v}$ is a bijective order preserving correspondence from $\mathcal C$ to $\text{Spec}(O_{v})$.

\end{cor}

As above, the proof of the previous corollary is quite similar to the proof of Corollary \ref{maximal chain covers specO_v}.

Some natural questions regarding quasi-valuation rings one may
consider now are: "Do we have a connection between $R$ being finitely generated over $O_{v}$ and $M_w$ being
cancellative"? "Is there an example of a finite dimensional field extension $E/F$
and a subring $R$ of $E$ lying over $O_{v}$ such that 
$R$ does not satisfy GU over $O_{v}$? In particular, $R$ is not finitely generated over $O_{v}$ and
there cannot exist a \qv $w$
extending $v$ on $F$ with $O_{w}=R$ and $M_{w}$ cancellative".
We shall answer these questions affirmatively. The answer to the first question will be presented
immediately and the answer to the second question will be given in a subsequent paper.







In Theorem \ref{GU} we proved that under certain assumptions on $R$ and $w$, if $M_{w}$ is cancellative then
$O_{w}$ satisfies GU over $O_{v}$. We also showed above that if
$R$ is finitely generated as a module over $O_{v}$ then $R$
satisfies GU over $O_{v}$. We will show now that if $R$ is torsion-free and
finitely generated as a module over $O_{v}$ then there exists a $v-$\qv
$w$ on $R$ extending $v$ with $R=O_{w}$ and $M_{w}$ cancellative.



\begin{rem} \label{minimal set of generators implies O_v-independent} Let $U$ be a torsion-free
$O_{v}-$module and let $C$ be a minimal set of generators of
$U$ over $O_{v}$. Then $C$ is $O_{v}-$independent and $| C | = \text{dim}_F (U \otimes_{O_{v}} F)$.
In particular, all minimal sets of generators of $U$ over $O_{v}$ have the same cardinality.
 \end{rem}

\begin{proof} Write $\sum_{i=1}^{k} \alpha_{i}x_{i}=0$ (for
$\alpha_{i} \in O_{v}, x_{i} \in C$). Assume to the contrary that
there exists $\alpha_{i} \neq 0$ and let $\alpha_{i_{0}}$ denote
an element with minimal $v-$value, $1 \leq i_{0} \leq k$. Then
$\alpha_{i_{0}}^{-1} \alpha_{i} \in O_{v}$ for all $1 \leq i \leq
k$ and, since $U$ is torsion-free over $O_{v}$, we get $$-x_{i_{0}}=\sum_{i \neq i_{0}} \alpha_{i_{0}}^{-1}
\alpha_{i}x_{i},$$ which contradicts the minimality of $C$.
Thus $C$ is $O_{v}-$independent, and therefore $\{ x \otimes 1 \}_{x \in C}$ is a basis of $U \otimes F$
over $F$.

\end{proof}

It is easy to see that the previous remark is not valid when replacing the valuation domain with a commutative ring,
and not even when replacing it with an integral domain (even when $U$ is finitely generated over $O_{v}$).
Also, the following statement does not hold: if $T \subseteq U$ is a maximal $O_{v}-$independent set then $T$ generates $U$
(as opposed to vector spaces over division rings).

In order to prove the existence of a $v-$\qv extending $v$ with a cancellative monoid,
in case $R$ is torsion-free and finitely generated as a module over $O_{v}$, we need the existence of a minimal set of generators containing 1,
as proved in the following lemma.


\begin{lem} \label{minimal set with 1} Let $R$ be a torsion-free $O_{v}-$algebra which
is finitely generated as a module over $O_{v}$. Then there exists
a minimal set of generators $B=\{ r_{1}, r_{2}, ...,r_{k}\}$ of $R$ over $O_{v}$, such
that $1 \in B$.

 \end{lem}

\begin{proof} Let $C= \{ x_{1}, x_{2}, ...,x_{k}\}$ be a set of
generators of $R$ over $O_{v}$ with $k$ minimal. 
Assume that $1 \notin C$ and
write

\begin{equation}
1= \sum_{i=1}^{k} \alpha_{i}x_{i} \tag{II} \label{3}
\end{equation}

for $\alpha_{i} \in
O_{v}$, $x_{i} \in C$. By Theorem \ref{Sa1, Theorem 9.19} and Lemma \ref{Sa1, Lemma 9.25}
there exists the filter \qv $w$ on $R$ satisfying $w(ar)=v(a)+w(r)$ for
every $a \in O_{v}$ and $r \in R$ (i.e., $w$ is a $v-$quasi-valuation with respect to $O_{v}$).
By Lemma \ref{R f.g implies R* cap S contained in S*}, $(R^{\times} \cap O_{v}) \subseteq O_{v}^{\times}$; hence, by Proposition \ref{almost equivalent conditions LO}, $w(1)=0$.

Thus, $$0=w(1)=w(\sum_{i=1}^{k} \alpha_{i}x_{i}) \geq \min_{1 \leq i
\leq k} \{ v(\alpha_{i})+w(x_{i}) \}.$$ Since $x_{i} \in R$ for
all $1 \leq i \leq k$, one has $w(x_{i}) \geq 0$ for all $1 \leq i
\leq k$. Therefore, one cannot have $v(\alpha_{i}) > 0$ for all $1
\leq i \leq k$. So, there exists $\alpha_{i_{0}}$ such that
$v(\alpha_{i_{0}})=0$, $1 \leq i_{0} \leq k$. Multiply equation
(\ref{3}) by $\alpha_{i_{0}}^{-1}$ and get
$-x_{i_{0}}=-\alpha_{i_{0}}^{-1} \cdot 1 +\sum_{i \neq i_{0}}
\alpha_{i_{0}}^{-1} \alpha_{i}x_{i}$. Let $B=(C \setminus \{
x_{i_{0}} \}) \cup \{1\}$; then $x_{i_{0}} \in sp ( B)$ and thus
$B$ is a set of generators of $R$ over $O_{v}$. Now, $|B|=k$ and $C$ is
a minimal set of generators of size $k$; therefore, by the previous Remark, $B$ is a minimal set of generators
of $R$ over $O_{v}$.



\end{proof}

\begin{thm} \label{R f.g. implies v-qv with Mw cancellative} Let $R$ be a torsion-free $O_{v}-$algebra which
is finitely generated as a module over $O_{v}$. Then there exists
a $v-$\qv $w$ on $R$ extending $v$ on $O_{v}$ such that $w(x) \neq \infty$ for all nonzero $x \in R$, and $M_{w}$ is
cancellative; moreover, $M_{w} = \Gamma_{v}$ and $O_{w}=R$.

 \end{thm}

\begin{proof} Let $B=\{ r_{1}=1, r_{2}, ...,r_{k}\}$ be a minimal set of generators of $R$ over $O_{v}$ containing
$1$ (there exists such a set by the previous Lemma). Define for
every $\sum_{i=1}^{k} \alpha_{i}r_{i} \in R$,
$$w(\sum_{i=1}^{k} \alpha_{i}r_{i})= \min_{1 \leq i
\leq k} \{ v(\alpha_{i}) \}.$$ $B$ is a minimal set of generators
of $R$ over $O_{v}$ and thus, by Remark \ref{minimal set of
generators implies O_v-independent}, $B$ is $O_{v}-$independent
and therefore $w$ is well defined. It is obvious that $w$ extends
$v$ (on $O_{v}$) and $w(x) \neq \infty$ for all nonzero $x \in R$. We shall now prove that $w$ is indeed a \qv on
$R$. Let $x=\sum_{i=1}^{k} \alpha_{i}r_{i}, y=\sum_{i=1}^{k}
\beta_{i}r_{i} \in R$. We have,
$$w(x+y)=w(\sum_{i=1}^{k} \alpha_{i}r_{i}+ \sum_{i=1}^{k}
\beta_{i}r_{i})=w(\sum_{i=1}^{k} (\alpha_{i}+\beta_{i})r_{i})=$$
$$\min_{1 \leq i \leq k} \{ v(\alpha_{i}+\beta_{i})\} \geq
\min_{1 \leq i \leq k} \{ \min \{v(\alpha_{i}),v(\beta_{i})\}\}=$$
$$\min \{ \min_{1 \leq i \leq k} \{ v(\alpha_{i})\}, \min_{1 \leq i \leq k} \{
v(\beta_{i})\}\}=$$$$ \min \{ w(\sum_{i=1}^{k} \alpha_{i}r_{i}),
w(\sum_{i=1}^{k} \beta_{i}r_{i}) \}=\min \{ w(x),w(y)\}.$$

Also, $$w(xy)=w(\sum_{i=1}^{k} \alpha_{i}r_{i} \cdot
\sum_{i=1}^{k} \beta_{i}r_{i})=w(\sum_{1 \leq i,j \leq k}
\alpha_{i} \beta_{j}r_{i}r_{j}) $$

Now, by the proof above $w(x'+y') \geq \min \{ w(x'),w(y')\}$ for
all $x',y' \in R$. Also, $r_{i}r_{j} \in R$ for every $1 \leq i,j
\leq k$. Thus,

 $$ w(\sum_{1 \leq i,j \leq k}
\alpha_{i} \beta_{j}r_{i}r_{j}) \geq \min_{1 \leq i,j \leq k}
\{w(\alpha_{i} \beta_{j}r_{i}r_{j}) \} .$$

Next, we show that for every $\alpha \in O_{v}$ and $r \in R$ we have
$w(\alpha r)=v(\alpha)+w(r)$. 
Indeed, let $\alpha \in O_{v}$ and write $r=\sum_{i=1}^{k}
\gamma_{i}r_{i}$ for $\gamma_{i} \in O_{v}$ and $r_{i} \in B$,
then $$w(\alpha \sum_{i=1}^{k} \gamma_{i}r_{i})=w( \sum_{i=1}^{k}
\alpha \gamma_{i}r_{i})=\min_{1 \leq i \leq k} \{ v(\alpha
\gamma_{i})\}$$$$=\min_{1 \leq i \leq k} \{ v(\alpha)+v(
\gamma_{i})\}=v(\alpha)+\min_{1 \leq i \leq k} \{ v(
\gamma_{i})\}$$$$=v(\alpha)+w(\sum_{i=1}^{k}
\gamma_{i}r_{i})=v(\alpha)+w(r).$$ Also, since $r_{i}r_{j} \in R$
for every $1 \leq i,j \leq k$, we get $w(r_{i}r_{j}) \geq 0$ for
every $1 \leq i,j \leq k$. Hence,

$$\min_{1 \leq i,j \leq k} \{w(\alpha_{i} \beta_{j}r_{i}r_{j}) \}=
\min_{1 \leq i,j \leq k} \{v(\alpha_{i}
\beta_{j})+w(r_{i}r_{j})\}$$$$ \geq \min_{1 \leq i,j \leq k}
\{v(\alpha_{i} \beta_{j}) \}=\min_{1 \leq i,j \leq k}
\{v(\alpha_{i})+v( \beta_{j})\}  $$
$$ =\min_{1 \leq i \leq k} \{v(\alpha_{i}) \}+ \min_{1 \leq i \leq
k} \{v(\beta_{i}) \}=w(\sum_{i=1}^{k}
\alpha_{i}r_{i})+w(\sum_{i=1}^{k} \beta_{i}r_{i})$$$$=w(x)+w(y).$$

\end{proof}

\begin{prop} \label{R f.g. implies natural extension with Mw cancellative} Let $R$ be a torsion-free $O_{v}-$algebra which
is finitely generated as a module over $O_{v}$. Then there exists
a $v-$\qv $W$ (with respect to $F$) on $R \otimes_{O_{v}} F$ such that
$w(x) \neq \infty$ for all nonzero $x \in R \otimes_{O_{v}} F$, $w$ extends $v$ on $F$,
$M_{W} = \Gamma_{v}$, and $O_{W}=R \otimes 1$. 
 \end{prop}

\begin{proof} By Theorem \ref{R f.g. implies v-qv with Mw cancellative} there exists a $v-$\qv $w$ (with respect to $O_{v}$) on
$R$ extending $v$ on $O_{v}$ such that $M_{w} = \Gamma_{v}$ and $w(x) \neq \infty$ for all nonzero $x \in R$;
$w$ is defined above. By Lemma \ref{[Sa1, Lemma 9.31]} (taking
$M=\Gamma_{v}$), there exists the natural extension $W$,
which extends $w$ on $R$ and such that $M_{W} = M_{w}$. Recall that $W$ is defined by
$W(r \otimes \frac{1}{b})=w(r)-v(b)$ for all nonzero $r
\otimes \frac{1}{b} \in R \otimes_{O_{v}} F$. It is easy to
see that $W$ is a $v-$\qv (with respect to $F$) on $R \otimes_{O_{v}} F$,
extending $v$ on $F$ and such that $w(x) \neq \infty$ for all nonzero $x \in R \otimes_{O_{v}} F$.  
We prove now that $O_{W}=R
\otimes 1$. Note that for every element $r \in R$, we have $w(r)
\geq 0$ and thus $W(r \otimes 1)=w(r) \geq 0$. Conversely,
let $r \otimes \frac{1}{b}  \in R \otimes_{O_{v}}F$, for
nonzero $r \in R$ and $b \in O_{v}$, with $W(r \otimes
\frac{1}{b}) \geq 0$. Write $r=\sum_{i=1}^{k} \alpha_{i}r_{i}$;
then $\min_{1 \leq i \leq k} \{ v(\alpha_{i}) \}=w(r) \geq v( b)$.
Thus, $\alpha_{i} \in bO_{v}$ for every $1 \leq i \leq k$. Hence,
$r=\sum_{i=1}^{k} bb^{-1} \alpha_{i}r_{i} \in bR$. So one can
write $r=b r'$ for some $r' \in R$. Therefore,
$$r \otimes \frac{1}{b}=b r' \otimes \frac{1}{b}=r'
\otimes 1.$$ Consequently, $O_{W}=R \otimes 1$.
\end{proof}



To clarify, we define the following properties,
for $R$ a torsion-free algebra over $O_{v}$:

(i) $R$ is finitely generated as a module over $O_{v}$;

(ii) There exists a $v-$\qv $w$ on $R$ such that $w(1)=0$, $w(x) \neq \infty$ for all nonzero $x \in R$, $M_{w}$ is cancellative
and $R \cong O_W$, where $W$ denotes the natural extension of $w$.
Moreover, there exists an $O_{v}-$subalgebra $R'$ of $R \otimes_{O_{v}} F$ such that $R'$ is left or
right artinian and contains $R$.

(iii) $R$ satisfies GU over $O_{v}$.

By Theorem \ref{R f.g. implies v-qv with Mw cancellative} and Proposition
\ref{R f.g. implies natural extension with Mw cancellative}, (i) implies (ii) (note that $[R \otimes_{O_{v}} F:F]<\infty$). By
Corollary \ref{If exist subalgebra such that and contains then GU}, (ii) implies (iii). It is easy to see that (ii) does not imply (i). Indeed, let $F$ denote the field of fractions of $O_{v}$,
let $E/F$ be any algebraic field extension such that $[E:F]= \infty$, and let $R$ be any valuation domain of $E$ lying over $O_{v}$. It is well known that there exists a valuation $u$ on $E$, extending $v$, whose value ring is $R$; in particular, $\Gamma_u$ is cancellative. It is clear that $R \otimes_{O_{v}} F$ is artinian whereas $R$ is not finitely generated as a module over $O_{v}$.

I do not know if (iii) implies (ii). However, it is easy to see that if
$R$ is not torsion-free over $O_{v}$ then (i) implies (iii) (by Corollary \ref{R f.g implies GU}), but does not imply (ii).
Indeed, let $R$ be an $O_{v}-$algebra, finitely generated as
a module over $O_{v}$, and is not faithful over $O_{v}$. Then by Remark \ref{w(1)=0 implies w extends v},
there cannot exist a $v-$\qv $w$ on $R$ such that $w(1)=0$.

Department of Mathematics, Sce College, Ashdod 77245, Israel.

{\it E-mail address: sarusss1@gmail.com}

\end{document}
